\documentclass{article}

\usepackage{amssymb}

\usepackage[english]{babel}
\usepackage{textcomp}
\usepackage{amsmath}
\usepackage{amsthm}
\usepackage{graphicx}

\newtheorem{defin}{Definition}[section]
\newtheorem{lemma}[defin]{Lemma}
\newtheorem{theor}[defin]{Theorem}

\newtheorem{prop}[defin]{Proposition}
\theoremstyle{remark}
\newtheorem{rem}{Remark}[section]

\newcommand{\E}{\mathbb{E}}   
\newcommand{\Pro}{\mathbb{P}} 

\newcommand{\N}{\mathbb{N}}   
\newcommand{\R}{\mathbb{R}}   
\newcommand{\Z}{\mathbb{Z}}   

\begin{document}

\begin{center}
\Huge
{Fractional L\'evy processes as a result of  compact interval integral transformation}
\end{center}

\thispagestyle{empty}
\normalsize

{Heikki Tikanm\"{a}ki, Aalto University, School of Science and Technology,\\Institute of Mathematics, P.O.Box 11100, FI-00076 Aalto, Finland.\\E-mail: heikki.tikanmaki@tkk.fi}

{Yuliya Mishura, Taras Shevchenko National University of Kyiv,  Ukraine. \\E-mail: myus@univ.kiev.ua}
\begin{abstract}
Fractional Brownian motion can be represented as an integral of a deterministic kernel w.r.t. an ordinary Brownian motion either on infinite or compact interval. In previous literature fractional L\'evy processes are defined by integrating the infinite interval kernel w.r.t. a general L\'evy process. In this article we define fractional L\'evy processes using the compact interval representation.

We prove that the fractional L\'evy processes presented via different integral transformations have  the same finite dimensional distributions if and only if they are fractional Brownian motions. Also, we present relations between  different fractional L\'evy processes and analyze the properties of such processes. A financial example is  introduced as well.


Subject classification (MSC2010): 60G22, 60G51.
\end{abstract}

\newpage
\section{Introduction}
\label{intro}
Fractional Brownian motion (fBm) has become very important tool in modern probability and statistical modeling. Fractional Brownian motion is defined as a Gaussian process with certain covariance structure. Besides of this definition, fBm could be represented equivalently as an integral of a deterministic kernel with respect to an ordinary Brownian motion. In fact, there exist at least two such kernels: Mandelbrot-Van Ness kernel that has infinite support and compactly supported Molchan-Golosov kernel.

Fractional Brownian motion allows to model dependency because of its covariance structure. Hence it is a popular model in many applications. There is one parameter, namely Hurst parameter $H$ that describes the whole dependence structure. For Hurst parameter $H>\frac{1}{2}$ the process has long-range dependence property. For $H<\frac{1}{2}$ the increments are negatively correlated and for $H=\frac{1}{2}$ the increments are independent i.e. we come to the ordinary Brownian motion case. Of course,  fBm has also several other properties such as self-similarity and stationarity of increments. Despite of all these properties, fBm is neither semimartingale nor Markov process (excluding the Brownian motion case $H=\frac{1}{2}$).

If one is interested in fBm because of its correlation structure, one might not need exactly fBm but just some process with the same covariance structure. The law of a Gaussian process is determined uniquely by its second order structure. However, if we drop the assumption of Gaussianity, then the covariance structure does not determine the law uniquely. Thus, there are several possible ways to generalize fractional Brownian motion to the case of fractional L\'evy processes. By choosing different ways of generalization, we preserve different properties of fBm.

In this paper we define fractional L\'evy processes by two different integral transformations. This means basically that we take the integral representation of fractional Brownian motion with respect to an ordinary Brownian motion and replace the driving Brownian motion with a general square integrable L\'evy process. The processes that we will end up with, share the covariance structure of fractional Brownian motion. However, these processes could be more flexible in modeling than fractional Brownian motion, since the driving L\'evy noise is more general than the Gaussian one. For example, we might be able to capture such a phenomenon as a shock in the stock market (jump of the driving L\'evy process) that affects the market with delay and has some long term impacts. The applications in different fields of science might also be possible.

Fractional L\'evy processes by Mandelbrot-Van Ness representation were first defined by~\cite{benassi}. The theory was developed further by~\cite{marquardt}. Molchan-Golosov transformation has been used in fractional L\'evy process setting for defining fractional subordinators in~\cite{bender}. The general definition for fractional L\'evy processes by Molchan-Golosov transformation is new to the best of our knowledge.

There are also several other related concepts in the literature. One of the best known are fractional stable motions, see~\cite{samorodnitsky}. However, fractional stable motions are not fractional L\'evy processes in the sense that they would share the covariance structure of fBm.

\section{Preliminaries}
Let $(\Omega,\mathcal{F},(\mathcal{F}_t)_{t\in \R},\Pro)$ denote a fixed filtered probability space.
\subsection{Definition and some properties of fBm}
\begin{defin}[Fractional Brownian motion]
Let $H\in(0,1)$. Fractional Brownian motion with Hurst parameter $H$ is a zero mean Gaussian process $B^H$ with covariance
\begin{equation*}
\E B^H_t B^H_s=\frac{1}{2}\left ( |t|^{2H}+|s|^{2H}-|t-s|^{2H}\right ).
\end{equation*}
\end{defin}
Besides of this definition, we can represent fractional Brownian motion in many equivalent ways. For example, we can choose integral representations with respect to an ordinary Brownian motion.

On one hand, let $(W_t)_{t\in \R}$ be two-sided Brownian motion, i.e. $W_t=W^{(1)}_t$, when $t\geq 0$ and $W_t=-W^{(2)}_{-t}$, when $t<0$. Here $W^{(1)},W^{(2)}$ are independent Brownian motions. Then it holds that
\begin{equation*}
\left(B^H_t\right )_{t\in \R}
{=}\left (\int_{-\infty}^t f_H(t,s)dW_s\right)_{t\in \R},
\end{equation*}
where the Mandelbrot-Van Ness kernel $f_H$ is given by
\begin{equation*}
f_H(t,s)=C_H\left ( (t-s)_+^{H-\frac{1}{2}}-(-s)_+^{H-\frac{1}{2}}\right), \quad s,t\in \R,
\end{equation*}
and the constant is given by
\begin{equation*}
C_H=\left ( \int_0^\infty \left ( (1+s)^{H-\frac{1}{2}}-s^{H-\frac{1}{2}}\right )^2ds+\frac{1}{2H}\right )^{-\frac{1}{2}}=\frac{\big(2H \sin \pi H
\Gamma(2H)\big)^{1/2}}{\Gamma(H+1/2)}.
\end{equation*}
On the other hand, fBm can be represented as well on compact interval by
\begin{equation*}
\left (B^H_t\right)_{t\geq 0}=\left (\int_0^tz_H(t,s)dW_s\right)_{t\geq 0},
\end{equation*}
where the Molchan-Golosov kernel $z_H$ is given by
\begin{equation*}
z_H(t,s)=c_H (t-s)^{H-\frac{1}{2}}F\left ( \frac{1}{2}-H,H-\frac{1}{2},H+\frac{1}{2},\frac{s-t}{s}\right), \quad 0< s< t<\infty,
\end{equation*}
and $z_H(t,s)=0$ otherwise. Here the Gauss' hypergeometric function $F$ of $x\in \R$ with parameters $a,b,c$ is defined by
\begin{equation*}
F(a,b,c,x)=\sum_{j=0}^\infty \frac{(a)_j(b)_j}{(c)_j}\frac{x^j}{j!},
\end{equation*}
where $(a)_0=0$ and $(a)_k=a\cdot(a+1)\dots (a+k-1)$ for $k\in \N$. The constant $c_H$ is given by
\begin{equation*}
c_H=\frac{1}{\Gamma(H+\frac{1}{2})}\left ( \frac{2H\Gamma(H+\frac{1}{2})\Gamma(\frac{3}{2}-H)}{\Gamma(2-2H)} \right)^{\frac{1}{2}}.
\end{equation*}
For $H>\frac{1}{2}$ we have the following simplified form of the kernel
\begin{equation*}
z_H(t,s)=\Big(H-\frac{1}{2}\Big)c_Hs^{\frac{1}{2}-H}\int_s^t u^{H-\frac{1}{2}}(u-s)^{H-\frac{3}{2}}du, \quad 0<s<t<\infty.
\end{equation*}
Note that we do not need the definition of any two-sided process for the compact interval Molchan-Golosov representation of fBm. For more details on the integral representations of fBm, see for example~\cite{jost} or~\cite{nualart}.
\subsection{L\'evy processes}
Consider now  the conventions related to L\'evy processes. By the well-known L\'evy-Khinchine theorem, the characteristic function of a L\'evy process $L$ at time $t\geq 0$ can be represented as
\begin{equation*}
\E e^{iuL_t}=\exp \left (t \Psi(u)\right ),
\end{equation*}
where the characteristic exponent $\Psi$ is given by
\begin{equation*}
\Psi(u)=i\gamma u-\frac{1}{2}\sigma^2u^2+\int_\R \left ( e^{iux}-1-iux1_{\{|x|<1\}}\right )\nu(dx),
\end{equation*}
with $\gamma \in \R$, $\sigma^2\geq 0$ and $\nu$ being a measure concentrated on $\R \backslash \{0\}$ and satisfying
\begin{equation*}
\label{levymeasurecondition}
\int_\R (1\wedge x^2)\nu(dx)<\infty;
\end{equation*}
see for instance~\cite{kyprianou}~or~\cite{sato}. We call $(\gamma,\sigma^2,\nu)$ the characteristic triplet of $L$. From now on we assume that $\E L_1^2<\infty$ and $\E L_1=0$. For simplicity we assume that there is no Gaussian component, i.e. $\sigma^2=0$. With these assumptions we see that the characteristic function can be written as
\begin{equation*}
\E e^{iuL_t}=\exp \left (t \int_{\R}\left (e^{iux}-1-iux\right )\nu(dx)\right ),
\end{equation*}
as in~\cite{bender2}.
\subsection{FLp by infinite interval transformation}
Fractional L\'evy processes by infinite interval transformation were defined for $H \in (\frac{1}{2},1)$ in~\cite{marquardt}. However,  the $L^2$- definition in~\cite{marquardt} can be extended for $H\in (0,\frac{1}{2})$ as well.

\begin{defin}[Two-sided L\'evy processes]
A two-sided L\'evy process or a L\'evy process on $\R$, $(L_t)_{t\in \R}$, is defined as $L_t=L^{(1)}_t$ if $t\geq 0$ and $L_t=-L^{(2)}_{-(t-)}$ if $t<0$, where $L^{(1)},L^{(2)}$ are independent and identically distributed L\'evy processes. We say that the characteristic triplet and exponent of $L^{(1)}$ are the characteristic triplet and exponent of $L$, respectively.
\end{defin}
\begin{defin}[FLp by Mandelbrot-Van Ness transformation]
\label{flpmvn}
Let $(L_t)_{t \in \R}$ be a L\'evy process defined on $\R$. Furthermore, assume that $\E L_1=0$ and $\E L_1^2<\infty$ and $L$ does not have Gaussian  component. For $H\in\left(0,1\right)$ we say that
\begin{equation*}
X_t=\int_{-\infty}^t f_H(t,s)dL_s
\end{equation*}
is the fractional L\'evy process by Mandelbrot-Van Ness transformation (fLpMvN), where the stochastic integral is understood as a limit in probability of elementary integrals or in $L^2$-sense.
\end{defin}
Also the following facts follow from~\cite{marquardt}:
\begin{itemize}
\item The integral can be understood pathwise if $H>\frac{1}{2}$ as an improper Riemann integral.
\item The paths of fLpMvN are continuous when $H>\frac{1}{2}$ and even H\"older continuous of order $H-\frac{1}{2}$ on compacts.
\item Fractional L\'evy processes are never self-similar. This is proved for the case $H>\frac{1}{2}$, but the same proof works for the whole range $H\in(0,1)$.
\end{itemize}
In this paper, we contribute to the theory of fLpMvN by proving Theorem~\ref{flpmvnqv} on quadratic variation of fLpMvN.
\section{FLp as a result of compact interval transformation}
The main contribution of this paper is the theory of fractional L\'evy processes obtained via compactly supported Molchan-Golosov transformation. Convenient  feature of  these processes is that we do not need their infinite history.
\begin{defin}[FLp by Molchan-Golosov transformation]
\label{flpmgdef}
Let $(L_t)_{t\geq 0}$ be a L\'evy process without Gaussian component such that $\E L_1=0$ and $\E L_1^2<\infty$. Let $H \in(0,1)$. We call the stochastic process
\begin{equation*}
Y_t=\int_0^t z_H(t,s)dL_s
\end{equation*}
fractional L\'evy process by Molchan-Golosov transformation (fLpMG to be short). Here $z_H$ is the Molchan-Golosov kernel.
\end{defin}
The definition is understood as taking the limit in probability of elementary integrals in the sense of~\cite{rajput}~or~\cite{marquardt}.
\begin{rem}
It is also possible to include Gaussian component by considering a sum of fLpMG driven by pure jump L\'evy process and an independent fBm with the same Hurst parameter.
\end{rem}
The following theorem is the main result of the paper. It basically states that Brownian motion is the only process with slight moment assumptions such that the both integral transformations give the same process (in distribution).
\begin{theor}\label{mainthm}
Let $H\in (\frac{1}{2},1)$ and $L$ be a (two-sided) L\'evy process with non-degenerate L\'evy measure s.t. $\E L_1=0$.

 1) If $\E |L_1|^3<\infty$ and $\E L_1^3\neq0$, then fLpMvN $X$ and fLpMG $Y$ driven by $L$ have different finite dimensional distributions.
 
 2) If $\E L_1^4<\infty$, then fLpMvN $X$ and fLpMG $Y$ driven by $L$ have different finite dimensional distributions.
\end{theor} 
The proof of this theorem is presented in section~\ref{mainproof}. Here it does not matter if the driving L\'evy process has Gaussian component or not because the proof is based on the 4th cumulant. In fact one does not need these moment assumptions, if the driving L\'evy process happens to be a compound Poisson process.
\begin{prop}
\label{mainimproved}
Let $L$ be a non-degenerate compound Poisson process s.t. $\E L_1=0$ and $\E L_1^2<\infty$. Then fLpMvN $X$ and fLpMG $Y$ have different finite dimensional distributions.
\end{prop}
Proof of the proposition is presented in section~\ref{mainproof}. A picture of the paths of the two fractional L\'evy processes is in Figure~\ref{flppathpic}. The driving L\'evy process is a compound Poisson process with jump sizes $\pm 1$. Note the different behavior near origin.
\begin{figure}
\begin{center}
\includegraphics[width=6cm]{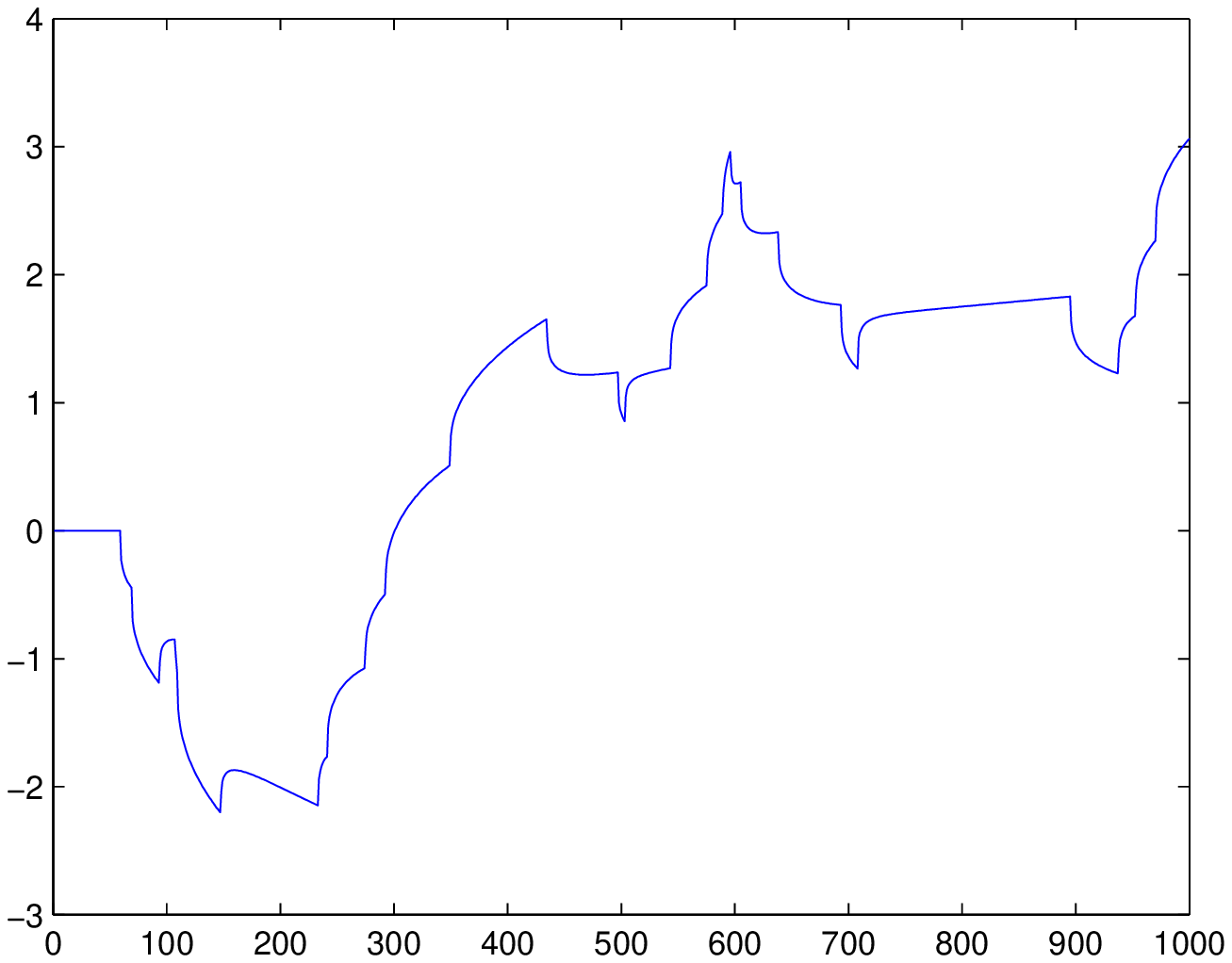}
\includegraphics[width=6cm]{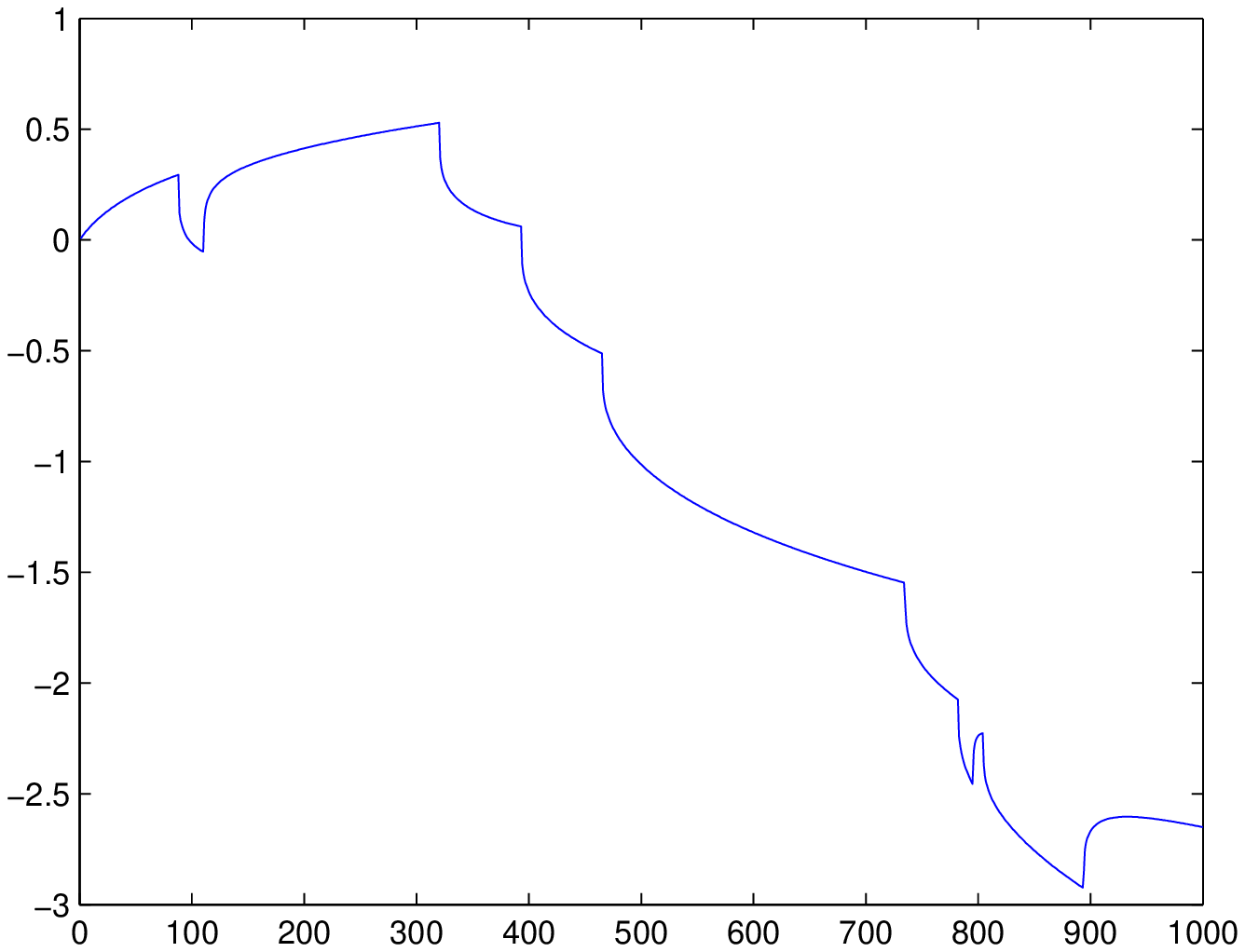}
\caption{A path of fLpMG (left) and fLpMvN (right) with L\'evy measure $\lambda(\delta_{-1}+\delta_{1})$ and $H=0.75$. The different behavior of the two processes can be seen near origin, even though the driving paths are not the same.}
\label{flppathpic}
\end{center}
\end{figure}
\begin{figure}
\begin{center}
\includegraphics[width=6cm]{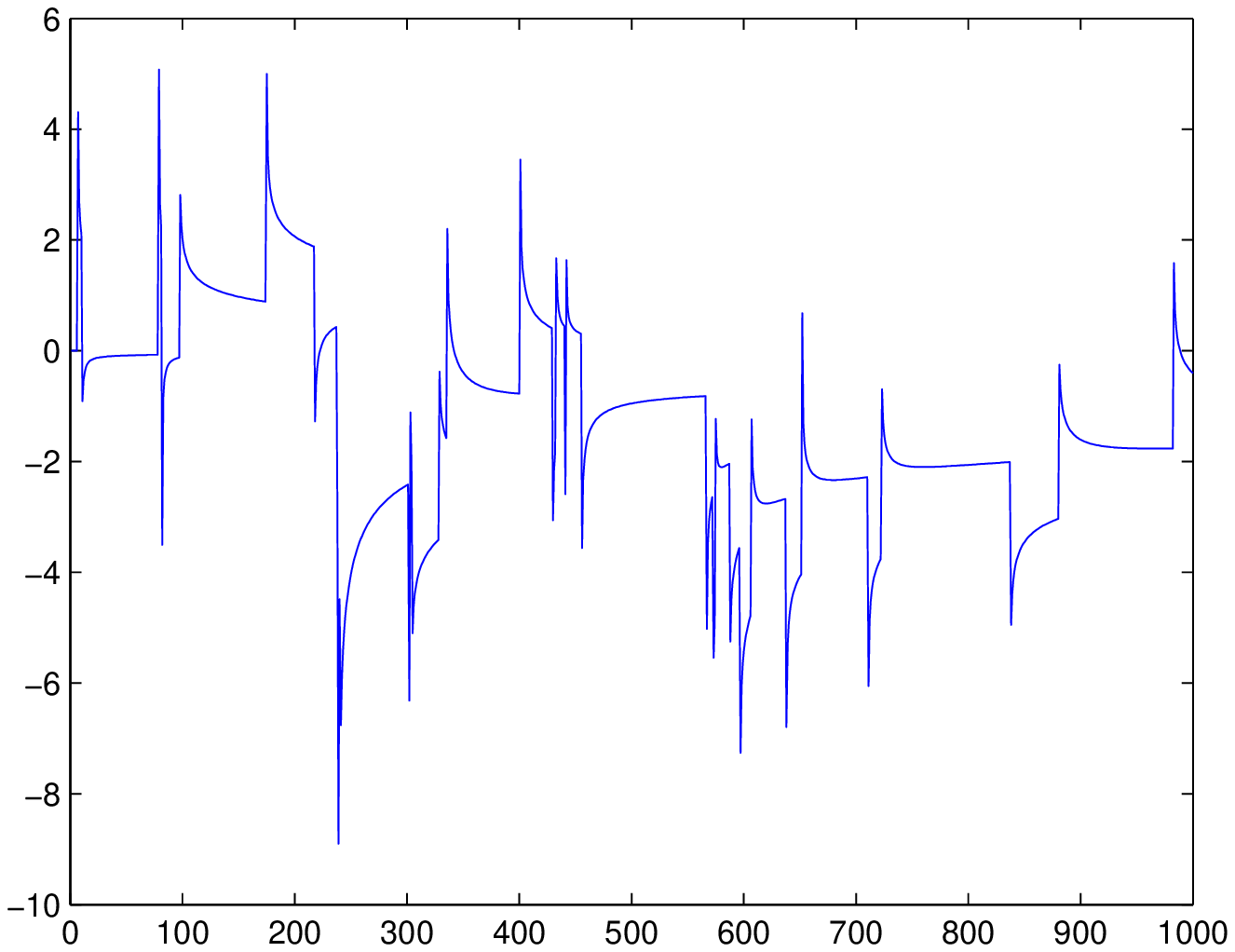}
\includegraphics[width=6cm]{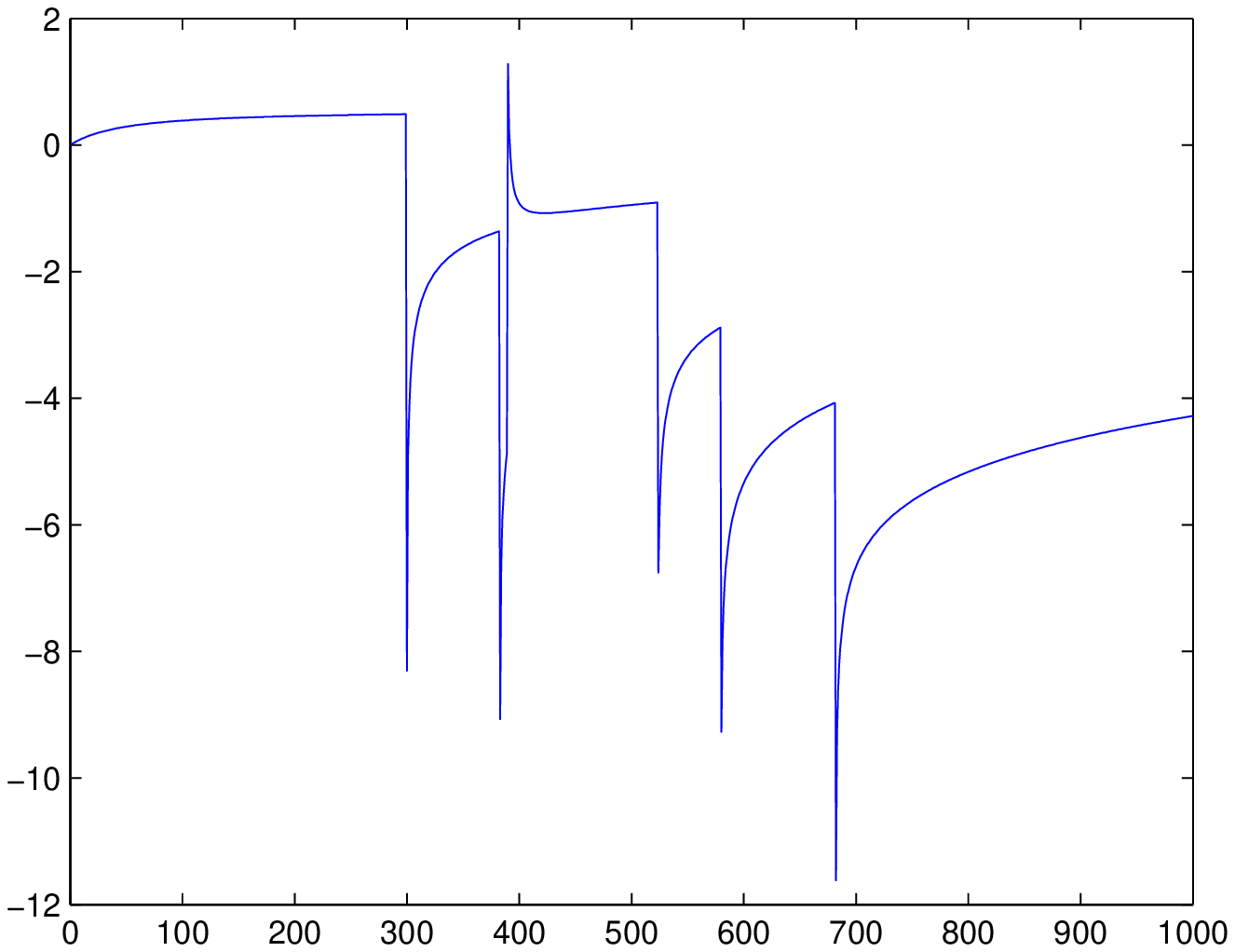}
\caption{A path of fLpMG (left) and fLpMvN (right) with $H=0.25$.}
\label{flppathpic2}
\end{center}
\end{figure}
\begin{rem}
An fLpMG driven by $L$ is adapted to the natural filtration $\mathcal{F}^L_t= \sigma \{L_u|0\leq u\leq t\}$. This is not the case for fLpMvN.
\end{rem}
\begin{prop}
An fLpMG $Y$ can be considered as the $L^2$- limit of approximating step functions. Moreover, the following $L^2$- isometry holds
\begin{equation*}
\E Y_t^2 = ||z_H(t,\cdot)||^2_{L^2([0,t])}\E L_1^2.
\end{equation*}
\end{prop}
\begin{proof}
This is a direct consequence of Proposition~2.1.~of~\cite{marquardt}.
\end{proof}
\begin{prop}[Autocovariance function]
\begin{equation*}
\E Y_tY_s=\frac{\E L_1^2}{2}\left ( t^{2H}+s^{2H}-|t-s|^{2H}\right ),
\end{equation*}
where $s,t\geq 0$.
\end{prop}
\begin{proof}
By $L^2$- isometry we have that
\begin{equation*}
\E Y_t^2 = ||z_H(t,\cdot)||^2_{L^2([0,t])}\E L_1^2.
\end{equation*}
We use the same argument for the increment (for the $L^2$- isometry, see \cite{marquardt} proposition~2.1.). Thus
\begin{equation*}
\E (Y_t-Y_s)^2=\E L_1^2\int_0^{t \vee s} (z_H(t,u)-z_H(s,u))^2du=\E L_1^2 \cdot |t-s|^{2H}.
\end{equation*}
Now
\begin{equation*}
\E Y_t Y_s=\frac{1}{2}\left ( \E Y_t^2 + \E Y_s^2 -\E (Y_t-Y_s)^2\right )=\frac{\E L_1^2}{2}\left (t^{2H}+s^{2H}-|t-s|^{2H}\right ).
\end{equation*}
\end{proof}
Besides the $L^2$- interpretation, we have also a partial result on the pathwise construction of fLpMG.
\begin{prop}[Pathwise construction]
\label{pathwise}
Let $H \in \left ( \frac{1}{2},1\right )$ and $L$ be a compound Poisson process with characteristic triplet $(0,0,\nu)$ such that $\E L_1=0$ and $\E L_1^2<\infty$. Then
\begin{equation*}
Y_t=\int_0^tz_H(t,s)dL_s=-\int_0^t \left ( \frac{d}{ds}z_H(t,s)\right )L_sds \quad \text{(almost surely)}.
\end{equation*}
\end{prop}
\begin{proof}
Fix $\omega \in \Omega$. A. s. there exists $\epsilon>0$ such that $L_s=0$ for all $s\in[0,\epsilon]$. The kernel $z_H(t,\cdot)$ is continuous when $s\neq 0$ by~\cite{jost}. Also its derivative is continuous on $(0,t)$. We get now by Lemma~2.1. of~\cite{eberlein} that
\begin{align*}
&\int_0^tz_H(t,s)dL_s=\int_{\epsilon}^t z_H(t,s)dL_s\\
=&z_H(t,t)L_t-z_H(t,\epsilon)L_\epsilon-\int_{\epsilon}^t \left ( \frac{d}{ds}z_H(t,s)\right )L_sds\\
=&-\int_0^t \left ( \frac{d}{ds}z_H(t,s)\right )L_sds.
\end{align*}
\end{proof}
The problem with the pathwise construction of fLpMG (when not in compound Poisson case) is that for $H>\frac{1}{2}$, the Molchan-Golosov kernel $z_H(t,\cdot)$ does not vanish at the origin like the Mandelbrot-Van Ness kernel does. However, the paths of the fLpMG are continuous when $H>\frac{1}{2}$ as is illustrated by the following theorem.

\begin{prop}
\begin{enumerate}
\item For $H>\frac{1}{2}$, an fLpMG $Y$ on $[0,T]$ has a. s. H\"older continuous paths of any order strictly less than $H-\frac{1}{2}$.
\item For $H<\frac{1}{2}$, an fLpMG $Y$ has discontinuous sample paths with positive probability.
\item For $H<\frac{1}{2}$, an fLpMG $Y$ has unbounded sample paths with positive probability.
\end{enumerate}
\end{prop}
\begin{proof}
Let $H>\frac{1}{2}$. It holds that
\begin{equation*}
\E |Y_t-Y_s|^2=|t-s|^{2H}=|t-s|^{1+2(H-\frac{1}{2})}.
\end{equation*}
The first assertion follows now from the Kolmogorov-Chentsov theorem. See for example~\cite{karatzasshreve}.

Let now $H<\frac{1}{2}$. We know that in this case the mapping $t\mapsto z_H(t,s)$ is unbounded and discontinuous for all $s\in(0,T)$. Thus by theorem~4 of~\cite{rosinski2} we know that the sample paths of $Y$ are unbounded with positive probability and also discontinuous with positive probability.
\end{proof}
\begin{rem}
Analogously one can prove that an fLpMvN has unbounded and discontinuous paths with positive probability when $H<\frac{1}{2}$.
\end{rem}
Besides continuity, the sample paths have also the zero quadratic variation property for $H>\frac{1}{2}$. This is illustrated in the following theorem where we compute the quadratic variation over the dyadic sequence of partitions.
\begin{theor}
\label{flpmgqv}
Let $(Y_t)_{t_{\geq 0}}$ be a fLpMG with $H>\frac{1}{2}$. Then for all $t>0$ it holds that
\begin{equation*}
\sum_{j=1}^{2^n} \left ( Y_{\frac{j}{2^n}t}-Y_{\frac{j-1}{2^n}t}\right )^2\rightarrow 0\quad \text{a. s. when } n \rightarrow \infty.
\end{equation*}
\end{theor}
\begin{proof}
Set
\begin{equation*}
V_n=\sum_{j=1}^{2^n}\left ( Y_{\frac{j}{2^n}t}-Y_{\frac{j-1}{2^n}t}\right )^2.
\end{equation*}
Now we have that
\begin{equation*}
\E V_n=(\E L_1^2) 2^n \left ( \frac{t}{2^n}\right )^{2H}=(\E L_1^2) t^{2H}2^{-n(2H-1)}.
\end{equation*}
We obtain using Markov inequality (see e.g.~\cite{jacodprotter}) that
\begin{align*}
&\sum_{n=1}^\infty \Pro\left (|V_n-\E V_n|\geq \frac{1}{n}\right )\leq \sum_{n=1}^\infty n\E |V_n-\E V_n|\\\leq &\sum_{n=1}^\infty n 2 \E V_n=2t^{2H} (\E L_1^2) \sum_{n=1}^\infty n 2^{-n(2H-1)}<\infty.
\end{align*}
We use now Borel-Cantelli theorem (see~\cite{jacodprotter}) and obtain that
\begin{equation*}
V_n -\E V_n\rightarrow 0, \quad \text{as } n \rightarrow \infty \text{ a. s.}
\end{equation*}
On the other hand $\E V_n \rightarrow_{n\rightarrow \infty}0$. Thus $V_n\rightarrow_{n\rightarrow \infty}0$ a. s.
\end{proof}
Note that the same proof works also in the case of fLpMvN. Thus, we have the following theorem.
\begin{theor}
\label{flpmvnqv}
Let $X$ be an fLpMvN with $H>\frac{1}{2}$. Then for $-\infty<s<t<\infty$ it holds that
\begin{equation*}
\sum_{j=1}^{2^n} \left ( X_{s+\frac{j}{2^n}(t-s)}-X_{s+\frac{j-1}{2^n}(t-s)}\right )^2 \rightarrow 0 \quad \text{a.s. when }n \rightarrow \infty.
\end{equation*}
\end{theor}
\begin{prop}[Characteristic function]
\label{mgflpcf}
Let $u_1,\dots,u_n \in \R$ and $0<t_1<\dots<t_n<\infty$. Then
\begin{equation*}
\E \left ( \exp \left ( i\sum_{j=1}^n u_j Y_{t_j}\right )\right )= \exp \left ( \int_\R \Psi \left ( \sum_{j=1}^n u_j z_H(t_j,s)\right )ds\right ),
\end{equation*}
where $\Psi$ is the characteristic exponent of the driving L\'evy process $L$. Moreover, $Y_t$ is infinitely divisible for all $t\geq 0$.
\end{prop}
\begin{proof}
This follows, for example, from~\cite{rajput}.
\end{proof}
\begin{prop}
The increments of fLpMG are not always stationary.
\end{prop}
\begin{proof}
Consider fLpMG $Y$ driven by a compound Poisson process $L$ with L\'evy measure $\nu=\frac{1}{2}(\delta_1+\delta_{-1})$, where $\delta_\cdot$ denotes the Dirac delta. For $\epsilon>0$
\begin{equation*}
\Pro(Y_\epsilon-Y_0=0)\geq \Pro(L_t=0 \quad\forall t\leq \epsilon)= 1-e^{-\epsilon}.
\end{equation*}
On the other hand, consider set
\begin{equation*}
A=\{\omega \in \Omega \text{ s.t. } \#\{s\in[0,1]\text{ s.t. }\Delta L_s\neq 0\}=1 \text{ and } \#\{s\in (1,\epsilon] \text{ s.t. }\Delta L_s\neq 0\}=0\}.
\end{equation*}
In set $A$, there is one jump time $S\in[0,1]$ and $\Delta L_S=\pm 1$. It follows that
\begin{equation*}
Y_{1+\epsilon}-Y_1=\left ( \frac{d}{ds}z_H(1+\epsilon,s)-\frac{d}{ds}z_H(1,s)\right )_{s=S}\Delta L_S\neq 0,
\end{equation*}
by proposition~\ref{pathwise}. We also have that
\begin{equation*}
\Pro(Y_{1+\epsilon}-Y_1=0)\leq 1-\Pro(A)=1-e^{-1}(1-e^{-\epsilon})<1-e^{-\epsilon},
\end{equation*}
when $\epsilon>0$ small enough. Thus, $Y_\epsilon-Y_0\stackrel{d}{\neq}Y_{1+\epsilon}-Y_1$. Hence, the increments of fLpMG $Y$ are not stationary.
\end{proof}
\begin{theor}[Self-similarity]
Fractional L\'evy process by Molchan-Golosov transformation cannot be self-similar for $H\in(0,1)$.
\end{theor}
\begin{proof}
Assume that the process $Y$ is self-similar with some index $\alpha$. Then we have for all $c>0$ that
\begin{equation*}
\left (Y_{ct}\right)_{t\geq 0}\stackrel{d}{=}\left(c^\alpha Y_t\right)_{t\geq 0}.
\end{equation*}
The characteristic function of $Y$ is given by theorem~\ref{mgflpcf}. On the other hand
\begin{align*}
&\E e^{iuY_t}=\E e^{iuc^{-\alpha}Y_{ct}}\\
=&\exp\left ( \int_0^{ct} \int_{\R}\left ( e^{ic^{-\alpha}uxz_H(ct,s)}-1-ic^{-\alpha}uxz_H(ct,s)\right)\nu(dx)ds  \right )\\
=&\exp\left ( \int_0^{t}\int_{\R}\left ( e^{ic^{-\alpha}c^{H-\frac{1}{2}}ux z_H(t,s)}-1-ic^{-\alpha}c^{H-\frac{1}{2}}uxz_H(t,s)\right )\nu(dx)cds\right )\\
=&\exp\left(\int_0^t \int_{\R}\left ( e^{iuxz_H(t,s)}-1-iuxz_H(t,s)\right )c\nu\left(c^{\alpha-H+\frac{1}{2}}dx\right)ds\right ).
\end{align*}
Note that $\alpha-H+\frac{1}{2}=0$ implies that $Y\stackrel{d}{=}cY$ for all $c>0$ which means that $Y=L=0$ identically. We define for $r>0$ the translation operator $T_r$ of measures on $\R$ for $B \in \mathcal{B}(\R)$ by
\begin{equation*}
(T_r\nu)(B)=\nu\left(r^{-1}B\right).
\end{equation*}
Now the L\'evy measure of infinitely divisible random variable $c^{-\alpha}Y_{ct}$ is given by $cT_r\nu_{Y_t}$, where $r=c^{-H+\frac{1}{2}+\alpha}$ and
\begin{equation*}
\nu_{Y_t}(B)=\int_0^t\int_{\R}1_B\left ( z_H(t,s)x\right )\nu(dx)ds.
\end{equation*}
The drift parameter of $Y_1$ is $\gamma_{Y_1}$. It follows from the uniqueness of the generating triplet and self-similarity property that it holds for all $b>0$
\begin{equation*}
\nu_{Y_t}=b^{1/\left(H-\frac{1}{2}-\alpha\right)}T_b\nu_{Y_t}.
\end{equation*}
Denote now $\beta=\frac{-1}{H-\frac{1}{2}-\alpha}$. Let $\mu$ be the distribution of random variable $Y_1$ and let $\hat{\mu}(u)$ be the characteristic function of $\mu$. Random variable $bY_1$ has now characteristic function $\hat{\mu}(bu)$. Because $Y_1$ is infinitely divisible, we can use proposition~11.10~of~\cite{sato} that the triplet of $bY_1$ is $(\gamma(b),0,T_b\nu_{Y_1})$ for some $\gamma(b)$. On the other hand, $\hat{\mu}(u)^{b^\beta}$ is an infinitely divisible characteristic function with triplet $(b^\beta\gamma_{Y_1},0,b^\beta\nu_{Y_1})$. Thus we have for any $b>0$ some $d$ such that
\begin{equation*}
\hat{\mu}(u)^{a}=\hat{\mu}(bu)e^{idu},
\end{equation*}
where $a=b^\beta$. We note that $\R_+\mapsto \R_+: b\mapsto b^\beta$ is one-to-one. Thus, $\mu$ follows a stable law with index $\beta$. The index $\beta \in (0,2]$ by definition~13.5.~of~\cite{sato}. By theorem~14.1.~of~\cite{sato}, $\beta=2$ corresponds to Gaussian case and is thus impossible.
It follows now that $\E Y_t^2=\infty$, which contradicts the fact that $Y_t\in L^2(\Omega,\Pro)$. Thus, fLpMG can never be self-similar of any order $\alpha$.
\end{proof}
\begin{rem}
In~\cite{bender}, the authors define fractional subordinators by Molchan-Golosov transformation using pathwise Riemann-Stieltjes integration. However, these processes are not fLpMG as considered here, since subordinators are increasing L\'evy processes and here we consider only zero mean L\'evy processes. Also the integration concept there is different.
\end{rem}
\section{Relation of the two fLp concepts}
The connection between fractional L\'evy processes by Molchan-Golosov transformation and Mandelbrot-Van Ness transformation is basically the same as in the fBm case. The result in fLp case is new.

Let $H\in\left( 0,1\right )$ and $L$ be a two-sided L\'evy process without Brownian component satisfying $\E L_1=0$ and $\E L_1^2<\infty$. Let $s>0$ and set
\begin{equation*}
Y^s_t=\int_0^tz_H(t,u)dL_{u-s}=c_H\int_0^t(t-u)^{H-\frac{1}{2}}\tilde{F}\left ( \frac{u-t}{u}\right)dL_{u-s}, \quad t \in [0,\infty),
\end{equation*}
which is in fact fLpMG with Hurst parameter $H$. Here
\begin{equation*}
\tilde{F}(x)=F\left ( \frac{1}{2}-H,H-\frac{1}{2},H+\frac{1}{2},x\right),
\end{equation*}
where $F$ is the Gauss' hypergeometric function. Define the time shifted process
\begin{equation*}
Z^s_t=Y^s_{t+s}-Y^s_s, \quad t\in[-s,\infty).
\end{equation*}
In the fBm case this would also be fBm, but in fLpMG case we do not have the stationarity of the increments and we are lacking such an interpretation. Now we substitute $v=u-s$ and obtain a.s. that
\begin{equation*}
Z^s_t=c_H\left ( \int_{-s}^t (t-v)^{H-\frac{1}{2}}\tilde{F}\left ( \frac{v-t}{v+s}\right )dL_v-\int_{-s}^0 (-v)^{H-\frac{1}{2}}\tilde{F}\left( \frac{v}{v+s}\right )dL_v\right ).
\end{equation*}
By~\cite{jost2}, $\tilde{F}(0)=1$ and thus we obtain formally as $s\rightarrow \infty$ that
\begin{equation*}
Z^\infty_t:=\frac{c_H}{C_H}X_t:=c_H \int_\R \left ( (t-v)_+^{H-\frac{1}{2}}-(-v)_+^{H-\frac{1}{2}}\right )dL_v, \quad t\in \R.
\end{equation*}
\begin{theor}
\label{flpmvnflpmgconnection}
For every $t\in \R$ there exist constants $S,C>0$ such that
\begin{equation*}
\E \left ( Z^s_t-Z^\infty_t\right )^2\leq Cs^{2H-2}, \quad \text{for } s>S.
\end{equation*}
\end{theor}
\begin{proof}
We obtain
\begin{align*}
&\frac{1}{c_H^2\E L_1^2}\E \left ( Z^s_t-Z^\infty_t\right )^2\\
=&\int_\R \left(\left((t-v)^{H-\frac{1}{2}}-(-v)^{H-\frac{1}{2}}\right)1_{(-\infty,-s)}(v)\right)^2dv\\
+&\int_\R \left((t-v)^{H-\frac{1}{2}}\left ( \tilde{F}\left ( \frac{v-t}{v+s}\right )-1\right)1_{(-s,t)}(v)\right.\\
-&\left.(-v)^{H-\frac{1}{2}}\left ( \tilde{F}\left ( \frac{v}{v+s}\right)-1\right)1_{(-s,0)}(v)\right)^2dv,
\end{align*}
by $L^2$- isometry and independence of increments of $L$. The claim follows now from the proof of Theorem~3.1.~of~\cite{jost2}.
\end{proof}
\section{Wiener integration}
Here our goal  is to define suitable Wiener integrals with respect to fLpMG. In contrary to the case of fLpMvN, we use the fractional integration on a compact interval instead of the whole real line. We will define the space $L^2_H([0,T])$ of integrands as in the case of compact interval Wiener integrals in fBm case.

Let $g$ be a function defined on $[0,T]$ and $I_-^{H-\frac{1}{2}}$ be the right-sided Riemann-Liouville integral operator of order $H-\frac{1}{2}$ as in~\cite{jost}. Define operator
\begin{equation*}
(K^Hg)(s)=\Gamma(H+\frac{1}{2})c_H s^{\frac{1}{2}-H}\left ( I_-^{H-\frac{1}{2}}\left ((\cdot)^{H-\frac{1}{2}}g(\cdot)\right )\right )(s), \quad s,t\in[0,T].
\end{equation*}
Now it holds by~\cite{jost} that
\begin{equation*}
z_H(t,s)=\Gamma(H+\frac{1}{2})c_Hs^{\frac{1}{2}-H}\left ( I_-^{H-\frac{1}{2}} (\cdot)^{H-\frac{1}{2}}1_{[0,t)}\right )(s),
\end{equation*}
for $0\leq t\leq T$. Define now the space
\begin{equation*}
L_H^2([0,T])=\{g \in L^1([0,T])| K^Hg\in L^2([0,T])\},
\end{equation*}
equipped with norm $||g||_{L_H^2([0,T])}:=||K^Hg||_{L^2([0,T])}$. Now we are ready for the definition of Wiener integral.
\begin{defin}[Wiener integral for fLpMG]
\label{wiflpmg}
Let $H\in(0,1)$, $Y$ be a fLpMG driven by L\'evy process $L$. For $g\in L^2_H([0,T])$ the Wiener integral with respect to fLpMG is defined as
\begin{equation*}
\int_0^T g(s)dY_s=\int_0^T (K^H g)(s)dL_s.
\end{equation*}
\end{defin}
Note that the definition is completely analogous to the definition of compact interval Wiener integrals in the fBm setup. Now, let $g$ be a step function, which means that
\begin{equation}
\label{stepfunction}
g(s)=\sum_{j=1}^n a_j 1_{(s_{j-1},s_j]}(s),
\end{equation}
where $a_0,\dots,a_n\in \R$ and $0=s_0<s_1<\dots<s_n=T$. Now $g \in {L}^{2}_{H}([0,T])$ and we have the following result.
\begin{lemma}
Assume $H\in(0,1)$, $Y$ is fLpMG driven by $L$ and $g$ a step function defined by equation~(\ref{stepfunction}). It holds that
\begin{equation*}
\int_0^T g(s)dY_s= \sum_{j=1}^n a_j \left(Y_{s_{j}}-Y_{s_{j-1}}\right).
\end{equation*}
\end{lemma}
\begin{proof}
It is clear from the definition that the integral of a step function is linear. We will prove the claim for indicator functions. The general claim follows from the linearity. Set $g(s)=1_{(s_1,s_2]}$. Now
\begin{align*}
&\int_0^T g(s)dY_s=\int_0^T \left ( K^H g\right )(s)dL_s=\int_0^T \left (K^H\left(1_{(0,s_2]}-1_{(0,s_1]}\right)\right )(s)dL_s\\
=&Y_{s_2}-Y_{s_1}.
\end{align*}
\end{proof}
Obviously the following isometry holds for a step function $g$
\begin{equation}
\label{mgintisometry}
||\int_0^T (K^Hg)(s)dL_s||_{L^2(\Pro)}^2=\E L_1^2\int_0^T (K^Hg)^2(s)ds=\E L_1^2 \cdot ||g||^2_{{L}^{2}_{H}([0,T])}.
\end{equation}
Next we restrict ourselves to the case $H\in\left(0,\frac{1}{2}\right)$. In this case $L^2_H([0,T])$ is complete (\cite{sottinen}) and the step functions are dense in $L^2_H([0,T])$. Thus we can make the following alternative definition. Note that both the definitions yield  the same Wiener integral.
\begin{defin}
\label{altdef}
Let $Y$ be a fLpMG with driving L\'evy process $L$ and Hurst index $H\in\left(0,\frac{1}{2}\right )$. Let $g\in L^2_H([0,T])$ and let $\{g_k\}_{k\in \Z_+}$ be a sequence of step functions converging to $g$ in $L^2_H([0,T])$. We define the Wiener integral of $g$ with respect to $Y$ as follows
\begin{equation*}
\int_0^T g(s)dY_s=L^2(\Pro)-\lim_{k\rightarrow \infty} \int_0^T g_k(s)dY_s.
\end{equation*}
\end{defin}
Note that the definition does not depend on the approximating sequence.
\section{Financial application}
Next we will construct an arbitrage free model including fractional L\'evy processes. This is a (geometric) mixed Brownian motion and fractional L\'evy process model. The no-arbitrage result is analogous to the result in the case of mixed Brownian motion and fractional Brownian motion.

In the following, $Z$ may be either fLpMG or fLpMvN with $H>\frac{1}{2}$ and $W$ is an ordinary Brownian motion independent of $Z$. Let $\sigma,\epsilon>0$. Define the mixed process by
\begin{equation}
\label{Uprocess}
U_t=\sigma Z_t +\epsilon W_t.
\end{equation}
\begin{theor}
\label{noarbitrage}
Let the market model be given by $(\Omega,\mathcal{F},\exp{U},(\mathcal{F}_t^U),\Pro)$ and let $\Phi$ be a stopping-smooth trading strategy, where we use the conventions of~\cite{bendersottinenvalkeila2}. Then $\Phi$ is not an arbitrage opportunity.
\end{theor}
\begin{proof}
We will check the assumptions of Theorem~5 of~\cite{bendersottinenvalkeila2} and then we are done. The two conditions to be checked, are the quadratic variation property and the conditional small ball property.

Both fLpMvN and fLpMG are continuous path processes with zero quadratic variation over the dyadic partitions, see Theorems~\ref{flpmgqv} and~\ref{flpmvnqv}. Thus $U$ has the quadratic variation of Brownian motion over these partitions.

Moreover, $U$ has conditional full support (CFS) w.r.t. its own filtration by Theorem~3.1 ~of~\cite{pakkanen}. Since $U$ has CFS w.r.t. $(\mathcal{F}^U_t)$ on $\R$, it follows that $\exp{U}$ has CFS w.r.t. $(\mathcal{F}^U_t)$ on $\R_+$. This is equivalent to the conditional small ball property of~\cite{bendersottinenvalkeila2} by Lemma~2.3 ~of~\cite{pakkanen}.
\end{proof}
The exact definition for stopping-smooth strategies is not given in this paper, because it is rather technical. According to~\cite{bendersottinenvalkeila2} the chosen strategies cover hedges for many European, lookback and Asian options. Thus, it is an economically meaningful class.

This mixed model is a natural way for modeling shocks in financial markets. The Brownian motion part corresponds to the ordinary noise in the market and the fractional L\'evy process part to sudden shocks in the market. On the other hand the fractional L\'evy process has the covariance structure of fBm, this allows to model for long-range dependence.

The no-arbitrage result holds for both fLp concepts, but from the modeling point of view they are different. If one wants to have stationary increments of $U$, one should use fLpMvN. If one wants to avoid history from $-\infty$, one should use fLpMG instead. In real world, there is always the time $0$ when the trading began. Hence fLpMG might be more natural choice. However, this modeling question is rather delicate.

If in the model~(\ref{Uprocess}), $Z$ is of bounded variation, then $U$ is a semimartingale with Brownian motion as the martingale part of the decomposition. However, this model has long-range dependence property.

\section{Proofs}
\label{mainproof}
First we prove a lemma about the connection of the normalizing constants of the different integral representations.
\begin{lemma} For any $H\in(0,1)$
\begin{equation*}
C_H=c_H.
\end{equation*}
\end{lemma}
\begin{proof}
First of all by~\cite{mishura}
\begin{equation*}
C_H=\frac{(2H \sin{(\pi H)}\Gamma{(2H)})^\frac{1}{2}}{\Gamma{(H+\frac{1}{2})}}.
\end{equation*}
Now we have that
\begin{align*}
&C_H-c_H=\frac{1}{\Gamma(H+\frac{1}{2})}\left ( (2H\sin{(\pi H)}\Gamma(2H))^\frac{1}{2}-\left (\frac{2H\Gamma(H+\frac{1}{2})\Gamma(\frac{3}{2}-H)}{\Gamma(2-2H)}\right )^{\frac{1}{2}}\right )\\
=&\frac{1}{\Gamma(H+\frac{1}{2})}\sqrt{\frac{2H}{\Gamma(2-2H)}}\left ( (\sin{(\pi H)}\Gamma(2H)\Gamma(2-2H))^\frac{1}{2}-(\Gamma(H+\frac{1}{2})\Gamma(\frac{3}{2}-H))^\frac{1}{2}\right ).
\end{align*}
For the difference we have now that
\begin{align*}
&(\sin{(\pi H)}\Gamma(2H)\Gamma(2-2H))^\frac{1}{2}-(\Gamma(H+\frac{1}{2})\Gamma(\frac{3}{2}-H))^\frac{1}{2}\\
=&(\sin{(\pi H)}(2H-1)\Gamma(2H-1)\Gamma(2-2H))^\frac{1}{2}-((H-\frac{1}{2})\Gamma(H-\frac{1}{2})\Gamma(\frac{3}{2}-H))^\frac{1}{2}\\
=&\sqrt{H-\frac{1}{2}}\left ( \left (2\sin{(\pi H)}\frac{\pi}{\sin{\pi(2H-1)}}\right )^\frac{1}{2}-\left( \frac{\pi}{\sin{\pi(H-\frac{1}{2})}}\right)^\frac{1}{2}\right )\\
=&\sqrt{\pi (H-\frac{1}{2})}\left ( \left ( -\frac{2\sin{(\pi H)}}{\sin{(2\pi H)}}\right )^\frac{1}{2}-\left ( \frac{-1}{\cos{\pi H}}\right )^\frac{1}{2}\right )\\
=&\sqrt{\pi (H-\frac{1}{2})\sin{2\pi H}\cos{(\pi H)}}\left ( (-2\sin{(\pi H)\cos{(\pi H)}})^\frac{1}{2}-(-\sin(2\pi H))^\frac{1}{2}\right )=0.
\end{align*}
The previous computation is for $H>\frac{1}{2}$, but an analogous computation goes through  for $H<\frac{1}{2}$ as well.
\end{proof}
Next we present some results about finiteness of the moments of different kernels.
\begin{lemma}
\label{ikm}
Let $H>\frac{1}{2}$ and $K>2$. Then for any $t>0$
\begin{equation*}
\int_0^t (z_H(t,s))^{K}ds=\infty,
\end{equation*}
when $H\geq \frac{1}{2}+\frac{1}{K}$.
\end{lemma}
\begin{proof}
\begin{align*}
&(H-\frac{1}{2})^{-K}c_H^{-K}\int_0^t (z_H(t,s))^{K}ds
=\int_0^t s^{K(\frac{1}{2}-H)}\left ( \int_s^t u^{H-\frac{1}{2}}(u-s)^{H-\frac{3}{2}}du\right )^{K}ds\\
\geq&\int_0^t s^{K(\frac{1}{2}-H)}\left ( \int_s^t (u-s)^{H-\frac{1}{2}}(u-s)^{H-\frac{3}{2}}\right )^K ds\\
=&\int_0^t s^{K(\frac{1}{2}-H)}\left ( \frac{1}{2H-1}(t-s)^{2H-1}\right )^Kds.
\end{align*}
We note that the factor
\begin{equation*}
\left ( \frac{1}{2H-1}(t-s)^{2H-1}\right )^K
\end{equation*}
is bounded and also bounded away from zero in some neighborhood of the origin. Thus the last integral is finite if and only if
\begin{equation*}
\int_0^t s^{K(\frac{1}{2}-H)}ds<\infty.
\end{equation*}
Thus the integral $\int_0^t (z_H(t,s))^Kds=\infty$ if $K(\frac{1}{2}-H)\leq 1$ i.e. $H\geq \frac{1}{2}+\frac{1}{K}$.
\end{proof}
\begin{lemma}
\label{fkm}
For $H>\frac{1}{2}$ and any  $t>0$
\begin{equation*}
\int_{-\infty}^t (f_H(t,s))^Kds<\infty, \quad K\geq 2.
\end{equation*}
\end{lemma}
\begin{proof} From self-similarity,  is sufficient to consider only $t=1$. We have that
\begin{align*}
&C_H^{-K}\int_0^1 (f_H(1,s))^Kds=\int_{-\infty}^0 \left ( (1-s)^{H-\frac{1}{2}}-(-s)^{H-\frac{1}{2}}\right )^Kds+\frac{1}{K(H-\frac{1}{2})+1}.
\end{align*}
For the first term we get
\begin{align*}
&\int_{-\infty}^0 \left ( (1-s)^{H-\frac{1}{2}}-(-s)^{H-\frac{1}{2}}\right )^Kds=\int_0^{\infty}\Big((1+s)^{H-\frac{1}{2}}-s^{H-\frac{1}{2}}\Big)^Kds\\
=& \int_0^1 z^{K(\frac{1}{2}-H)}((1+z)^{H-\frac{1}{2}}-1)^Kz^{-2}dz=\int_0^1 z^{K(\frac{1}{2}-H)-2}((1+z)^{H-\frac{1}{2}}-1)^Kdz.
\end{align*}
By the Lagrange  theorem we have that for some $x\in[0,z]$
\begin{equation*}
(1+z)^{H-\frac{1}{2}}-1=\Big(H-\frac{1}{2}\Big)(1+x)^{H-\frac{3}{2}}z\leq \Big(H-\frac{1}{2}\Big)z.
\end{equation*}
Now we have that
\begin{align*}
&\int_0^1 z^{K(\frac{1}{2}-H)-2}((1+z)^{H-\frac{1}{2}}-1)^Kdz\leq\Big(H-\frac{1}{2}\Big)^K\int_0^1 z^{K(\frac{1}{2}-H)-2}z^Kdz\\
=&\Big(H-\frac{1}{2}\Big)^K \int_0^1 z^{K(\frac{3}{2}-H)-2}dz<\infty,
\end{align*}
since $K\left (\frac{3}{2}-H\right)-2>2\left(\frac{3}{2}-1\right)-2=-1.$

\end{proof}
\begin{rem} It follows from two above lemmas  that for any $H\geq \frac{1}{2}+\frac{1}{K}$ and $t>0$ we have inequality
$\int_{-\infty}^t (f_H(t,s))^Kds<\int_0^t (z_H(t,s))^Kds$.
\end{rem}
Now we want to establish similar inequalities for $K=3, 4$ and $\frac{1}{2}< H< \frac{1}{2}+\frac{1}{K}$.
\begin{lemma}
\label{mgbound}
1) For any $\frac{1}{2}< H< \frac{5}{6}$ and any $t>0$ we have the inequality

$\int_{-\infty}^t (f_H(t,s))^3ds<\int_0^t (z_H(t,s))^3ds$.

2) For any $\frac{1}{2}< H< \frac{3}{4}$ and any $t>0$ we have the inequality

$\int_{-\infty}^t (f_H(t,s))^4ds<\int_0^t (z_H(t,s))^4ds$.

\end{lemma}
\begin{proof} The proof is similar in both the cases so consider the first one. It is better to normalize the integrals, and for the normalized Molchan-Golosov  kernel we obtain that
\begin{align*}
&\frac{1}{C_H^3}\int_0^1(z_H(1,s))^3ds\\
=&\Big(H-\frac{1}{2}\Big)^3\int_0^1 s^{\frac32-3H}\left (\int_s^1 u^{H-\frac{1}{2}}(u-s)^{H-\frac{3}{2}}du\right)^3ds.\\
\end{align*}
Note that integration by parts leads to the equality
\begin{equation*}
\Big(H-\frac{1}{2}\Big)\int_s^1 u^{H-\frac{1}{2}}(u-s)^{H-\frac{3}{2}}du=(1-s)^{H-\frac{1}{2}}-\Big(H-\frac{1}{2}\Big)\int_s^1 u^{H-\frac{3}{2}}(u-s)^{H-\frac{1}{2}}du.
\end{equation*}
Further, for $s\leq u\leq 1$ we have that
\begin{equation*}
(u-s)^{H-\frac{1}{2}}=u^{H-\frac{1}{2}}\Big(1-\frac{s}{u}\Big)^{H-\frac{1}{2}}\leq u^{H-\frac{1}{2}}(1-s)^{H-\frac{1}{2}},
\end{equation*}
whence
\begin{align*}
&(1-s)^{H-\frac{1}{2}}-\Big(H-\frac{1}{2}\Big)\int_s^1 u^{H-\frac{3}{2}}(u-s)^{H-\frac{1}{2}}du\\\geq&(1-s)^{H-\frac{1}{2}}
\Big(1-\Big(H-\frac{1}{2}\Big)\int_s^1u^{2H-2}du\Big)=\frac{1}{2}(1-s)^{H-\frac{1}{2}}(1+s^{2H-1}).
\end{align*}
Denote $p=H-\frac12, \, \frac12< H < \frac56$ and $\alpha_p^1=\frac{1}{C_H^3}\int_0^1(z_H(1,s))^3ds.$
Then $p\in(0,\frac13)$ and we obtain from previous estimates that
\begin{align*}
&\alpha_p^1\geq\frac{1}{8}\int_0^1s^{-3p}(1-s)^{3p}(1+s^{2p})^3ds\\
=&\int_0^1s^{-3p}(1-s)^{3p}s^{3p}ds+\frac{1}{8}\int_0^1s^{-3p}(1-s)^{3p}((1+s^{2p})^3-8s^{3p})ds\\
&=\int_0^1(1-s)^{3p}ds+\frac{1}{2}\int_0^1s^{-3p}(1-s)^{3p}(1-s^p)^{2}\Big(\frac{1+2s^{2p}+s^{4p}}{4}+\frac{s^p+s^{3p}}{2}+s^{2p}\Big)ds\\
&=:\alpha_p^{11}+\alpha_p^{12}.
\end{align*}

Evidently, $\alpha_p^{11}=\frac{1}{3p+1}$. For $\alpha_p^{12}$
 we use the following simple bounds: on the interval $s\in[0,1]$ $1\geq s^{kp}, k=1,2,3,4$
 and for $p\in(0,\frac13),\,s\in[0,1]$ $(1-s)^{3p}\geq 1-s^{3p}.$

Therefore, \begin{align*}
&\alpha_p^{12}\geq\frac32\int_0^1s^p(1-s^{3p})(1-s^p)^{2}ds\\
&=\frac32\int_0^1(s^p-2s^{2p}+s^{3p}-s^{4p}+2s^{5p}-s^{6p})ds\\
&=\frac32\Big(\frac{1}{p+1}-\frac{2}{2p+1}+\frac{1}{3p+1}-\frac{1}{4p+1}+\frac{2}{5p+1}-\frac{1}{6p+1}\Big)\\
&=9p^3\frac{38p^2+21p+3}{\prod_{k=1}^6(kp+1)}.
\end{align*}

On the other hand, for the  normalized  Mandelbrot-Van Ness kernel we can use the same reasonings as in the proof of Lemma \ref{fkm} and obtain in the above notations that
 
\begin{align*}
&\alpha_p^2:=C_H^{-3}\int_0^1 (f_H(1,s))^3ds\\
&=\int_{-\infty}^0 \left ( (1-s)^{p}-(-s)^{p}\right )^3ds+\frac{1}{3p+1}\\
&=\int_0^1 s^{-3p-2}((1+s)^{p}-1)^3ds+\frac{1}{3p+1}\\
&\leq p^3\int_0^1 s^{-3p+1}ds+\frac{1}{3p+1}=\frac{p^3}{2-3p}+\frac{1}{3p+1}.
\end{align*}

To establish the inequality $\alpha_p^{1}>\alpha_p^{2}$, that is equivalent to the statement of the lemma, it is sufficient to prove that
$$9p^3\frac{38p^2+21p+3}{\prod_{k=1}^6(kp+1)}>\frac{p^3}{2-3p}, \,p\in \Big(0, \frac13\Big),$$
or $$9\frac{38p^2+21p+3}{\prod_{k=1}^6(kp+1)}>\frac{1}{2-3p}, \,p\in \Big(0, \frac13\Big).$$

For technical simplicity, diminish the left-hand side and compare the functions $f_1(p)=9\frac{36p^2+21p+3}{\prod_{k=1}^6(kp+1)}=27\frac{12p^2+7p+1}{\prod_{k=1}^6(kp+1)}$ and $f_2(p)=\frac{1}{2-3p}$. Of course, to finish the proof, it is sufficient to establish the inequality$f_1(p)\geq f_2(p),\,p\in(0,\frac13)$.

Note that $f_2(p)$ increases from $\frac12$ to $1$. The derivative of $f_1(p)$ can be estimated, up to positive  multiplier, as 
\begin{align*}
&(24p+7)\prod_{k=1}^6(kp+1)-(12p^2+7p+1)\sum_{r=1}^6 r\prod_{k=1, k\neq r}^6(kp+1)\\
&\leq \prod_{k=1}^5(kp+1)((24p+7)(6p+1)-21(12p^2+7p+1))\\
&=-\prod_{k=1}^5(kp+1)(108p^2+81p+14)<0.
\end{align*}
It means that $f_1(p)$ decreases on $(0,\frac13),$ and it decreases from $27$ to $\frac{243}{160}>1$ that completes the proof. 
\end{proof}
\begin{rem} The integral with respect to Molchan--Golosov kernel can be bounded from below in terms  of Beta- or Gamma-functions. These estimates are more sharp that obtained in the proof of Lemma \ref{mgbound}; however, we can not proceed with them otherwise than numerically.
Indeed, for example, in the case  $K=4$ we can estimate \begin{align*}
&\frac{1}{C_H^4}\int_0^1(z_H(1,s))^4ds\geq\frac{1}{16}\int_0^1s^{-4p}(1-s)^{4p}(1+s^{2p})^4ds\\
&= \frac{1}{16}B(3-4H,4H-1)
+\frac{1}{4}B(2-2H,4H-1)+\frac{1}{4}B(2H,4H-1)\\&+\frac{1}{16}B(4H-1,4H-1)+\frac{3}{8}\frac{1}{4H-1}=:g_1(p).
\end{align*}
As before, $\frac{1}{c_H^4}\int_0^1(f_H(1,s))^4ds\leq\frac{p^4}{5-4p}+\frac{1}{4p+1}=g_2(p).$
The difference $g_1(p)-g_2(p), p=H-\frac12,$ is presented on Figure \ref{dif} in terms of Hurst parameter $H$.
\begin{figure}
\begin{center}
\includegraphics[width=6cm]{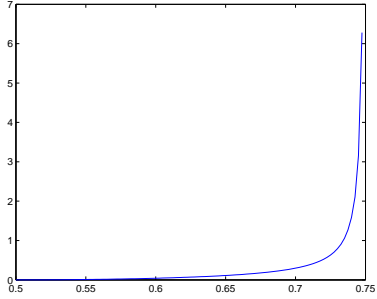}
\caption{The difference of for the integrals of Molchan-Golosov and Mandelbrot-van-Ness kernels for $K=4$ as the function of $H$.}
\label{dif}
\end{center}
\end{figure}
\end{rem}
Now we are ready for the proof of the main result, i.e.,  Theorem \ref{mainthm}.
\begin{proof}
With assumptions $\E L_1=0$ and $\E L_1^2<\infty$ we can write the characteristic exponent of the driving L\'evy process as follows
\begin{equation*}
\Psi(t)=-\frac{1}{2}\sigma^2 u^2+\int_{\R} \left ( e^{iux}-1-iux \right )\nu(dx).
\end{equation*}
Prove only the second statement, the first one can be  proved in a similar way. We use the representation formula for the characteristic function of fractional L\'evy process and get that the fourth cumulant of the fLpMG $Y$ is given by
\begin{equation*}
\int_0^t (z_H(t,s))^4ds \int_\R x^4 \nu(dx).
\end{equation*}
Analogously the fourth cumulant of the fLpMvN $X$ is
\begin{equation*}
\int_{-\infty}^t (f_H(t,s))^4ds \int_\R x^4\nu(dx).
\end{equation*}
Note that iff the L\'evy measure $\nu$ is nondegenerate, $\int_\R x^4\nu(dx)>0$.

Our aim is now to prove that with the assumption $\E L_1^4<\infty$ the fourth cumulants of different fLp's are different. This will prove that the different fractional L\'evy processes have different distributions.

For $H\geq \frac{3}{4}$ the 4th cumulant of fLpMG $Y$ is infinite by lemma~\ref{ikm}. On the other hand the corresponding cumulant for fLpMvN $X$ is finite by lemma~\ref{fkm}.

The case $H\in(\frac{1}{2},\frac{3}{4})$ uses Lemma \ref{mgbound}, and we immediately obtain the proof. 
\end{proof}
Next we proof Proposition~\ref{mainimproved}.
\begin{proof}
Let $L$ be a compound Poisson process with parameter $\lambda$ s.t. $\E L_1=0$ and $\E L^2_1<\infty$. We have for fLpMG $Y$ by proposition~\ref{pathwise} $\Pro (Y_t=0)\geq e^{-\lambda t}$ for $t>0$. On the other hand, for $t>0$ we can decompose fLpMvN $X$ to two independent components
\begin{equation*}
\int_{-\infty}^{-1}f_H(t,s)dL_s+\int_{-1}^tf_H(t,s)dL_s.
\end{equation*}
Probability that $L$ jumps exactly once on $[-1,t)$ is $\lambda(1+t)e^{-\lambda (1+t)}$ and probability that $L$ does not have jumps on $[-1,t)$ is $e^{-\lambda(1+t)}$. Now it is easy to see that
\begin{equation*}
\Pro (X_t=0)\leq \lambda(1+t)e^{-\lambda (1+t)} \wedge e^{-\lambda(1+t)} < e^{-\lambda t}\leq\Pro(Y_t=0)
\end{equation*}
for $t>0$ small enough.
\end{proof}
\section*{Acknowledgements}
We would like to thank Esko Valkeila for his fruitful comments. H. Tikanm\"aki has been supported financially by \textit{The Finnish Graduate School in Stochastics and Statistics} and \textit{Academy of Finland}, grant 212875. Yu. Mishura was supported financially by Finnish Academy of Science and Letters, Vilho, Yrj\"o and Kalle V\"ais\"al\"a Foundation.
\bibliographystyle{plain}

\begin{thebibliography}{10}

\bibitem{benassi}
Benassi,~A., Cohen,~S., and Istas,~J. 2004.
\newblock On roughness indices for fractional fields.
\newblock {\em Bernoulli} 10(2):357--373.

\bibitem{marquardt}
Marquardt, T. 2006.
\newblock Fractional {L}\'evy processes with an application to long memory
  moving average processes.
\newblock {\em Bernoulli} 12(6):1099--1126.

\bibitem{bender}
Bender, C., and Marquardt, T. 2009.
\newblock Integrating volatility clustering into exponential {L}\'evy models.
\newblock {\em J. Appl. Probab.} 46(3):609--628.

\bibitem{samorodnitsky}
Samorodnitsky, G., and Taqqu, M.S. 1994.
\newblock {\em Stable non-{G}aussian random processes},
\newblock Chapman \& Hall, Boca Raton.

\bibitem{jost}
Jost, C. 2007.
\newblock {\em Integral Transformations of Volterra Gaussian Processes},
\newblock PhD Thesis, University of Helsinki, Helsinki.

\bibitem{nualart}
Nualart, D. 2006.
\newblock {\em The Malliavin Calculus and Related Topics},
\newblock Springer, Berlin.

\bibitem{kyprianou}
Kyprianou, A.E. 2006.
\newblock {\em Introductory Lectures on Fluctuations of L\'{e}vy Processes with
  Applications},
\newblock Springer, Berlin.

\bibitem{sato}
Sato, K.-I. 1999.
\newblock {\em L\'{e}vy Processes and Infinitely Divisible Distributions},
\newblock Cambridge University Press, Cambridge.

\bibitem{bender2}
Bender, C., and Marquardt, T. 2008.
\newblock Stochastic calculus for convoluted {L}\'evy processes.
\newblock {\em Bernoulli} 14(2):499--518.

\bibitem{rajput}
Rajput, B.S., and Rosinski, J. 1989.
\newblock Spectral representations of infinitely divisible processes.
\newblock {\em Probab. Th. Rel. Fields} 82:451--487.

\bibitem{eberlein}
Eberlein, E., and Raible S. 1999.
\newblock Term structure models driven by general {L}\'evy processes.
\newblock {\em Math. Finance} 9(1):31--53.

\bibitem{karatzasshreve}
Karatzas I., and Shreve, S.E. 1998.
\newblock {\em Brownian Motion and Stochastic Calculus}.
\newblock Springer, New York.

\bibitem{rosinski2}
Rosinski, J. 1989.
\newblock On path properties of certain infinitely divisible processes.
\newblock {\em Stochastic Process. Appl.} 33:73--87.

\bibitem{jacodprotter}
Jacod, J., and Protter, P. 2004.
\newblock {\em Probability Essentials}.
\newblock Springer, Berlin.

\bibitem{jost2}
Jost, C. 2008.
\newblock On the connection between {M}olchan-{G}olosov and {M}andelbrot-{V}an
  {N}ess representations of fractional {B}rownian motion.
\newblock {\em J. Integral Equations Appl.} 20(1):93--119.

\bibitem{sottinen}
Sottinen, T. 2003.
\newblock {\em Fractional Brownian motion in finance and queuing}.
\newblock PhD Thesis, University of Helsinki, Helsinki.

\bibitem{bendersottinenvalkeila2}
Bender, C., Sottinen, T., and Valkeila, E. 2008.
\newblock Pricing by hedging and no-arbitrage beyond semimartingales.
\newblock {\em Finance Stoch.} 12(4):441--468.

\bibitem{pakkanen}
Pakkanen, M.S. 2010.
\newblock Stochastic integrals and conditional full support.
\newblock {\em J. Appl. Probab.} to appear.

\bibitem{mishura}
Mishura, Yu. 2008.
\newblock {\em Stochastic Calculus for Fractional Brownian Motion and Related
  Processes}.
\newblock Springer, Berlin.

\end{thebibliography}

\end{document}